\documentclass[12pt,reqno]{amsart}
\usepackage[numbers,sort&compress]{natbib}
\usepackage{amsfonts,bbm}
\usepackage{amssymb,color}
\usepackage{amssymb}
\usepackage{amsmath,amssymb}
\usepackage{fancyhdr}
\usepackage[titletoc]{appendix}
\usepackage{enumitem}
\usepackage[colorlinks=true,backref]{hyperref}
\hypersetup{urlcolor=blue, citecolor=blue, linkcolor=black}
\usepackage[left=3.5cm, right=3.5cm, lines=45, top=0.8in, bottom=1.0in]{geometry}  
\usepackage{amsfonts}
\usepackage{bbm}
\usepackage{amsmath,amssymb}
\usepackage{mathrsfs}
\usepackage{hyperref}
\usepackage{graphicx}
\usepackage{float}
\usepackage{amsmath,amscd,amsbsy,amssymb,latexsym,url,bm,amsthm}
\usepackage{mathrsfs}
\usepackage{hyperref}
\usepackage{graphicx}
\usepackage{float}
\usepackage{color}
\usepackage{amsmath,amscd,amsbsy,amssymb,latexsym,url,bm,amsthm}
\usepackage[vlined,boxed,commentsnumbered,ruled]{algorithm2e}
\usepackage{epsfig,graphicx,subfigure}
\usepackage{enumitem,balance,mathtools}
\usepackage{wrapfig}
\usepackage{mathrsfs, euscript}
\usepackage[usenames]{xcolor}
\usepackage{caption}
\usepackage{lipsum}
\usepackage{siunitx}		
\usepackage{bibentry}
\usepackage{tikz} 		
\usepackage{verbatim}	
\usepackage{amsaddr}

\newtheorem{thm}{Theorem}[section]
\newtheorem{cor}[thm]{Corollary}
\newtheorem{lemma}[thm]{Lemma}
\newtheorem{prop}[thm]{Proposition}
\newtheorem{defn}[thm]{ \bf{Definition}}
\theoremstyle{remark}
\newtheorem{rem}[thm]{Remark}

\newtheorem{exa}[thm]{Example}

\makeatletter \renewenvironment{proof}[1][Proof]
{\par\pushQED{\qed}\normalfont\topsep6\p@\@plus6\p@\relax\trivlist\item[\hskip\labelsep\bfseries#1\@addpunct{.}]\ignorespaces}{\popQED\endtrivlist\@endpefalse} \makeatother
\makeatletter

\newcommand{\Cas}[1]{\begin{cases*} #1 \end{cases*}}
\def\ncase#1{\begin{align*}
		\left\{\begin{aligned}#1\end{aligned}\right.
	\end{align*}	}	

\newcommand{\Del}[1]{}

\newcommand{\norm}[1]{\left\|#1\right\|}		
\newcommand{\abs}[1]{\left|#1\right|}			
\newcommand{\set}[1]{\left\{#1\right\}}			
\newcommand{\inner}[1]{\left\langle#1\right\rangle} 
\newcommand{\rmnum}[1]{\uppercase\expandafter{\romannumeral#1}}  

\def\norm#1{\left\|#1\right\|}

\def\abs#1{\left|#1\right|}

\def\brk#1{\left(#1\right)}			

\def\pd{\partial}						

\def\sch{\mathscr{S}}

\def\mpq{M_{p,q}}

\def\mtwone{\mathscr{M}_{2,1}}	 	      
\def\mtwonemu{\mathscr{M}_{2,1}^{\mu}}
\def\mpqmu{\mathscr{M}_{p,q}^{\mu}}
\def\mpqj{M_{p,q}^{[j]}}

\def\mpqinfty{\mathfrak{M}_{p,q,\infty}^{\rho}}

\def\eps{\varepsilon}
\def\leq{\leqslant}
\def\geq{\geqslant}

\newcommand{\N}{{\mathbb N}}			

\newcommand{\C}{{\mathbb C}}
\newcommand{\Z}{{\mathbb Z}}

\DeclareMathOperator*{\esssup}{ess\,sup}
\newcommand{\supp}{{\mbox{supp}\ }}

\def\al{\alpha}

\def\th{\theta}
\def\si{\sigma}
\def\de{\delta}

\def\la{\lambda}
\def\ga{\gamma}
\def\real{{\mathbb{R}}}			

\def\rev#1{\frac{1}{#1}}

\def\FF{{\mathscr{F}^{-1}}}		
\def\F{{\mathscr{F}}}			
\numberwithin{equation}{section}

\begin{document}
\title[Well-posedness of KdV and d-4NLS in $\mtwonemu$]{Well-posedness of  generalized KdV and one-dimensional fourth-order derivative nonlinear Schr\"odinger  equations for data with an infinite $L^2$ norm}


\author{Yufeng Lu}
\address{School of Mathematical Sciences,
	Peking University,
	Beijing 100871, China}
\thanks{}



\date{}

\begin{abstract}\noindent 
We study the Cauchy problem for the generalized KdV and one-dimensional fourth-order derivative nonlinear Schr\"odinger  equations, for which the global well-posedness of solutions with the small rough data in certain scaling limit of modulation spaces $\mtwonemu$ is shown, which contain some data with infinite $L^{2}$ norm.
\end{abstract}

\maketitle


\section{Introduction}

In this paper, We consider the generalized KdV and fourth-order derivative nonlinear Schr\"odinger  equations on real line as follows. 
\begin{align}
	\label{eq-d4nls}&i\pd_{t}u + \pd_{x}^{4}u = \lambda\pd_{x} \brk{u^{m+1}}, \  u(0,x) = u_{0}(x),\\
	\label{eq-kdv}&\pd_{t} u + \pd_{x}^{3} u = \lambda \pd_{x} (u^{m+1}), \ \ \ u(0,x)=u_{0}(x).
\end{align}
Here $u(t,x)$ is a complex-value function of $(t,x) \in \real^{1+1}, i= \sqrt{-1}, \la \in \C, \pd_{t}=\pd/\pd t, \pd_{x} = \pd/\pd x, \pd_{x}^{n} = \pd^{n}/\pd x^{n}$ for $n=3,4$. $m$ is an integer. We will study the global well-posedness of \eqref{eq-d4nls} and \eqref{eq-kdv} with small rough data in scaling limit of modulation spaces $\mtwonemu$.

The modulation space $\mpq^{s}$ is one of the function spaces, introduced by Feichtinger \cite{Feichtinger2003Modulation} in the 1980s using the short-time Fourier transform to measure the decay and the regularity of the function differently from the usual $L^{p}$ Sobolev spaces or Besov-Triebel spaces. Roughly speaking,  Besov-Triebel spaces mostly use the dyadic decompositions of the frequency space, while the modulation spaces use the uniform decompositions. Modulation spaces also have many applications in the analysis of partial differential equations. For example, the Schr\"odinger and wave semigroups, which are not bounded on neither $L^{p}$ nor $B_{p,q}^{s}$ for $p\neq 2$, are bounded on $\mpq^{s}$  (see \cite{Benyi2007Unimodular}). So, the modulation space is a good space for the initial data of the Cauchy problem for nonlinear dispersive equations (see \cite{Baoxiang2006Isometric,Wang2007global,Wang2013Globally,Nicola2014Phase,Kato2017Well,Oh2021global}). However, Sugimoto and Tomita in \cite{Sugimoto2007dilation} got the sharp scaling property of $\mpq$. Their result showed that the modulation spaces did not have good scaling properties like $L^{p}$ spaces. Therefore, in order to obtain some kinds of modulation spaces with good scaling property,  B$\acute{e}$nyi and Oh in \cite{Benyi2020Modulation}, also Sugimoto and  Wang in \cite{Sugimoto2021Scaling} introduced the scaling limit of modulation spaces like $\mpqinfty$ and $\mpqmu$. They studied the basic properties of these spaces such as the scaling property, the dual space, and the algebraic property. And they got some applications in NLS. In this paper, we obtain some other applications of the scaling limit of modulation spaces in nonlinear dispersive equations.

Our main results are as follows.

\begin{thm}
	\label{thm-GWP}
	Let $m\geq 8, p= m/2, A= \frac{3(p-4)}{2(3p+2)}, 0\leq \mu \leq \frac{3(m-8)}{2(3m+4)}$. Then there exists $\eps>0$, such that for any $u_{0} \in \mtwonemu$, with $\norm{u_{0}}_{\mtwonemu} \leq \eps$, then \eqref{eq-d4nls} has a unique global solution $u\in C(\real, \mtwonemu) \cap X^{\mu-A} (L_{x}^{p+2/3} L_{t}^{3p+2})$, where $X^{\mu-A} (L_{x}^{p+2/3} L_{t}^{3p+2})$ will be defined in Definition \ref{def-work-space}. Moreover, the data-to-solution map above is Lipschitz continuous.
\end{thm}

\begin{rem}
	By the inclusion relation between modulation spaces and the scaling limit of modulation spaces (Proposition 3.2 in \cite{Chen2022embedding}), we know that when $0\leq \mu \leq 3(m-8)/(6m+8)$, we have  $M_{2,1} \hookrightarrow \mtwonemu \hookrightarrow M_{p_{0},1}$, where $p_{0}=(3m+4)/14$. Also, by Proposition \ref{prop-mpqmu-embed-lr}, when $\mu>0$, $\mtwonemu \hookrightarrow L^{2}$ is not true, which means that there exists some $u_{0}\in \mtwonemu$ with arbitrarily large $L^{2}$ norm, such that \eqref{eq-d4nls} is global well-posedness with initial date $u_{0}$. Moreover, by Example \ref{exa:L2-infty}, the $L^{2}$ norm of $u_{0}$ could even be infinite, which means that $\mtwonemu$ is $L^{2}$ supercritical.
\end{rem}

The fourth-order nonlinear Schr\"odinger equations (4NLS) have been studied by many authors. For equations without derivative on nonlinear term, the well-posedness theory in $H^{s}$ has been studied in \cite{Segata2007asymptotic,Seong2021Well,Pausader2009cubic,Wang2012Nonlinear}. Hayashi et al. \cite{Hayashi2016Scattering,Hayashi2016Asymptotics,Hayashi2021Scattering} studied the well-posedness and scattering theory in weighted Sobolev spaces. For equations with first order derivatives on nonlinear terms (d-4NLS) such as \eqref{eq-d4nls}, the scaling invariant homogeneous Sobolev space is $\dot{H}^{s_{c}}$, where $s_{c}=1/2-3/m$. That is to say, for any solution $u(t,x)$ of \eqref{eq-d4nls} with initial data $u_{0}$, the scaling function $u_{\la}(t,x)=\la^{3/m}u(\la^{4}t,\la x)$ is also a solution of \eqref{eq-d4nls} with initial data $u_{0,\la}(x)= \la^{3/m}u_{0}(\la x)$, and satisfies \begin{align*}
	\norm{u_{0,\la}}_{\dot{H}^{s_{c}}} = \norm{u_{0}}_{\dot{H}^{s_{c}}}.
\end{align*}
Wang in \cite{Wang2012Global} proved the global well-posedness of \eqref{eq-d4nls} with small initial data in $H^{s_{c}}$ when $m\geq 4$. The cases of $m=2,3$ and high dimensional cases of (d-4NLS) have been studied by Hirayama and Okamoto in \cite{Hirayama2016Well}. The well-posedness theory in weighted Sobolev spaces was studied by Hayashi et al. in \cite{Hayashi2015Large,Hayashi2015Global}. For equations with higher order derivatives on nonlinear terms, local well-posedness in $H^{s}, s\geq 1/2$ was obtained by Segata in \cite{Segata2003Well}. One can also refer to \cite{Huo2005Cauchy,Huo2007refined}. For the high-dimensional case, Hao et al. in \cite{Hao2007Well} proved local well-posedness in $H^{s}, s>d+2+1/2$. Ruzhansky et al. in \cite{Ruzhansky2016Global} obtained the global well-posedness in $M_{2,1}^{s}$ for $s\geq 3+1/m$.


For the generalized KdV equation \eqref{eq-kdv}, our main result is as follows.

\begin{thm}\label{thm-kdv}
	Let $m\geq 4,A=(m-4)/(2m+2), 0\leq \mu \leq A$. Then there exists $\eps>0$ such that for any $u_{0}\in \mtwonemu$ with $\norm{u_{0}}_{\mtwonemu} \leq \eps$, then \eqref{eq-kdv} has a unique global solution $u\in C(\real,\mtwonemu)\cap X^{\mu-A}(L_{x}^{m+1}L_{t}^{2(m+1)})$. Moreover, the data-to-solution map above is Lipschitz continuous.
\end{thm}

\begin{rem}
	By the inclusion between modulation spaces and the scaling limit of modulation spaces, we have $M_{2,1} \hookrightarrow \mtwonemu \hookrightarrow M_{2(m+1)/5,1}$. Therefore, Theorem \ref{thm-kdv} is a generalization of the corresponding result of Theorem 1.2 in \cite{Wang2007Frequency}. Meanwhile, the proof of Theorem \ref{thm-kdv} is similar to the proof of Theorem \ref{thm-GWP}. we omit it for simplicity.
\end{rem}

By the same scaling argument as for \eqref{eq-d4nls}, we know that the critical Sobolev spaces of \eqref{eq-kdv} is $\dot{H}^{s_{c}}, s_{c}=1/2-2/m$. Recall that Kenig et al. in \cite{Kenig1993Well} showed the global well-posedness of \eqref{eq-kdv} with small initial data in $\dot{H}^{s_{c}}$ when $m\geq 4$. Later, they proved the ill-posedness of focusing KdV in $H^{s}$ with $s<s_{c}$ in \cite{Kenig2001ill}. Molinet and Ribaud in \cite{Molinet2003Cauchy}, extended the global well-posedness result in $\dot{B}_{2,\infty}^{s_{c}}$. Wang and Huang in \cite{Wang2007Frequency} proved the global well-posedness of \eqref{eq-kdv} in $M_{2,1}$, which was the case of $\mu=0$ in Theorem \ref{thm-kdv}. For the cases of $m=1,2,3$, many authors have been studied on it. One can refer to \cite{Colliander2003Sharp,HarropGriffiths2020Sharp,Killip2018KdV,Oh2021global}.

The paper is organized as follows. In Section \ref{sec:pre}, we will give notations and definitions of some function spaces, including the scaling limit of modulation spaces. In Section \ref{sec: inclusion}, we will give some characterization of scaling limit of modulation spaces and (Fourier) Lebesgue spaces. The proof of our main result will be given in Sections \ref{sec:linear}-\ref{sec:proof}.

\vskip 1.5cm

\section{Preliminary}\label{sec:pre}

\vskip .5cm
\subsection{Notation}

We write $\sch(\real^{d})$ to denote the Schwartz space of all complex-valued rapidly decreasing infinity differentiable functions on $\real^{d}$, and $\sch'(\real^{d})$ to denote the dual space of $\sch(\real^{d})$, all called the space of all tempered distributions. For simplification, we omit $\real^{d}$ without causing ambiguity. The Fourier transform is defined by $\F f(\xi) = \hat{f}(\xi) = \int_{\real^{d}} f(x) e^{-ix\xi} d\xi$, and the inverse Fourier transform by $\FF f(x) = (2\pi)^{-d} \int_{\real^{d}}f(\xi) e^{ix\xi} d\xi$. For $1\leq p <\infty$, we define the $L^{p}$ norm:
\begin{align*}
	\norm{f}_{p} = 	\left( \int_{\real^{d}} \abs{f(x)}^{p}dx\right )^{1/p}
\end{align*}
and $\norm{f}_{\infty} = \esssup_{x\in \real^{d}} \abs{f(x)}$.

We use the notation $I \lesssim J$ if there is an independent constant C such that $I \leq C J$. Also, we denote $I \approx J$ if $I \lesssim J$ and $J\lesssim I$. For $1\leq p \leq \infty$, we denote the dual index $p'$ with $1/p+1/p'=1$. Let $W(t)= e^{it \pd_{x}^{4}}= \FF e^{it \xi^{4}} \F $, the fourth-order Schr\"odinger semigroup; $\mathscr{A}f= \int_{0}^{t} W(t-s) f(s) ds.$

\begin{defn}
	\label{def-work-space}
	Let $(A,\norm{ \cdot}_{A})$ be a Banach space. For $\mu \in \real$, denote \begin{align*}
		X^{\mu}(A):= \set{u \in \sch' (\real^{2}): u = \sum_{j\leq 0} u_{j}, \mbox{ with } \sum_{j\leq 0} 2^{j\mu} \sum_{k\in \Z} \norm{\Box_{j,k} u_{j}}_{A} < \infty},
	\end{align*}
where the norm of this space is as follows.
\begin{align*}
	\norm{u}_{X^{\mu} (A)} := \inf \sum_{j\leq 0} 2^{j\mu} \sum_{k\in \Z} \norm{\Box_{j,k} u_{j}}_{A},
\end{align*}
where the infimum is taken over all the decompositions of $u=\sum_{j\leq 0}u_{j} \in X^{\mu}(A)$.
\end{defn}

\begin{rem}
	\label{rm-Xmu-embed}
	By the definition above, we could easily obtain $X^{\mu} (A) \hookrightarrow X^{\mu+\eps} (A) $ for any $\eps>0$.
\end{rem}

\subsection{Modulation spaces and scaling limit of modulation spaces}
Recall that the short time Fourier transform of $f$ respect to a window function $g\in \sch$ is defined as (see \cite{Feichtinger2003Modulation}):
\begin{align*}
	V_{g}f(x,\xi) = \int_{\real^{d}} f(t) \overline{g(t-x)} e^{-it\xi} dt.
\end{align*}
And then, for $1\leq p,q\leq \infty$, we denote 
\begin{align*}
	\norm{f}_{\mpq} = \norm{\norm{V_{g}f(x,\xi)}_{L_{x}^{p}}}_{L_{\xi}^{q}} := \norm{V_{g}f(x,\xi)}_{L_{\xi}^{q}L_{x}^{p}}.
\end{align*}
The modulation space $\mpq$ is defined as the space of all tempered distributions $f\in 
\sch'$ for which $\norm{f}_{\mpq}$ is finite. 

Also, we know another equivalent definition of modulation spaces by uniform decomposition of the frequency space (see \cite{Groechenig2001Foundations,Benyi[2020]copyright2020Modulation}).

Let $\sigma$ be a smooth cut-off function adapted to the unit cube $[-1/2,1/2]^{d}$ and $\sigma=0$ outside the cube $[-3/4,3/4]^{d}$, we write $\sigma_{k} = \sigma(\cdot -k)$, and assume that 
\begin{align*}
	\sum_{k\in\Z^{d}} \sigma_{k} (\xi) \equiv 1, \ \forall \xi \in \real^{d}.
\end{align*}

Denote $\sigma_{k} (\xi) =\sigma (\xi-k)$, and $\Box_{k} = \FF \sigma_{k} \F $, then we have the following equivalent norm of modulation space:

\begin{align*}
	\norm{f}_{\mpq} = \norm{\norm{\Box_{k}f}_{L_{x}^p}}_{\ell_{k}^{q} (\Z^{d})}.
\end{align*}

Then let us recall the definition of the scaling limit of modulation spaces (see \cite{Sugimoto2021Scaling}). Denote $\sigma_{j,k}(\xi) =\sigma(2^{-j} \xi  -k)$, and $\Box_{j,k}=\FF\sigma_{j,k}\F$, then we denote
\begin{equation}
	\norm{f}_{\mpq^{[j]}}:= \norm{\norm{\Box_{j,k}f}_{L_{x}^p}}_{\ell_{k}^{q} (\Z^{d})}.
\end{equation}
The scaling limit of modulation spaces $\mpqmu$ are defined as follows.
\begin{align*}
	\mpqmu = \set{ f\in \sch': \exists f_{j} \in \mpqj \mbox{ such that } f= \sum_{j\leq 0} f_{j}, \sum_{j\leq 0} 2^{j\mu} \norm{f_{j}}_{\mpqj} <\infty}
\end{align*}
and the norm on $\mpqmu$ is defined as \begin{align*}
	\norm{f}_{\mpqmu} = \inf \sum_{j\leq 0} 2^{j\mu} \norm{f_{j}}_{\mpqj},
\end{align*}
where the infimum is taken over all the decompositions of $f=\sum_{j\leq 0} f_{j} \in \mpqmu.$

\begin{rem}
	\label{rm-mpqmu-mpq}
	By the trivial decomposition $f=f+0+\cdots$, we know $\mpq \hookrightarrow \mpqmu$. For this reason, we could regard $\mpqmu$ as an extension of $\mpq$.
\end{rem}

\subsection{Useful lemmas}
In this subsection, we gather some useful results.

\begin{lemma}[Bernstein's inequality, Lemma 6.1 in \cite{Wang2011Harmonic}]
	\label{lem-bernstein}
	Let $1\leq p \leq q \leq \infty, R>0,\xi_{0} \in \real^{d}$. Denote $$L_{B(\xi_{0},R)}^{p} = \set{ f\in L^{p}: \mbox{ supp } \hat{f} \subseteq B(\xi_{0},R)}.$$ Then there exists $C>0$, such that \begin{align*}
		\norm{f}_{q} \leq C R^{d(1/p-1/q)} \norm{f}_{p}
	\end{align*} 
	hold for all $f\in L_{B(\xi_{0},R)}^{p}$ and $C$ is independent of $R>0$ and $\xi_{0} \in \real^{d}$.
\end{lemma}

\begin{lemma}[Christ–Kiselev, Lemma 2 in \cite{Molinet2004Well}]
	\label{lem-christ}
	Let $T$ be a linear operator defined in space-time functions $f(t,x)$ by \begin{align*}
		Tf(t)= \int_{\real} K(t,s)f(s)ds,
	\end{align*}
such that \begin{align*}
	\norm{Tf}_{L_{x}^{p_{1}}L_{t}^{q_{1}}} \lesssim \norm{f}_{L_{x}^{p_{2}}L_{t}^{q_{2}}},
\end{align*}
where $p_{2}\vee q_{2} < p_{1}\wedge q_{1}$. Then 
\begin{align*}
	\norm{\int_{0}^{t} K(t,s) f(s)ds}_{L_{x}^{p_{1}}L_{t}^{q_{1}}} \lesssim \norm{f}_{L_{x}^{p_{2}}L_{t}^{q_{2}}}.
\end{align*}
\end{lemma}

\begin{lemma}[Lemma 4.1 in \cite{Chen2022embedding}]
	\label{lem-mpq-lr}
	Let $1\leq p,q,r\leq \infty$. Then $\mpq \hookrightarrow L^{r}$ if and only if $p\vee q \leq r \leq q'$.
\end{lemma}

\begin{lemma}[Lemma 5.2 in \cite{Chen2022embedding}]
	\label{lem-mpq-Flr}
	Let $1\leq p,q,r\leq \infty$. Then $\mpq \hookrightarrow \F L^{r}$ if and only if $p\leq 2, q\leq r\leq p' $.
\end{lemma}

\vskip 1.5cm
\section{Inclusion between $\mpqmu$ and $L^{r}$}\label{sec: inclusion}

\begin{prop}
	\label{prop-mpqmu-embed-lr}
	Let $1\leq p,q,r\leq \infty$, denote \begin{align*}
		a(p,q) &= d(1/p+1/q-1),\\
		\si(p,q) &= 0\wedge d(1/q-1/p) \wedge a(p,q). 
	\end{align*}
	When $\si(p,q) \leq \mu \leq a(p,q), \si(p,q) <a(p,q)$, then $\mpqmu \hookrightarrow L^{r}$ if and only if $p\vee q \leq r \leq q', \mu-d/p\leq -d/r.$
\end{prop}

\begin{proof}
	Sufficiency: for any $f=\sum_{j\leq 0} f_{j}, f_{j} \in \mpq^{[j]}$. When $p\vee q \leq r \leq q'$, by Lemma \ref{lem-mpq-lr}, we have $\mpq \hookrightarrow L^{r}$. Therefore, when $\mu-d/p+d/r \leq 0$, by the triangular inequality, we have 
	\begin{align*}
		\norm{f}_{r} &\leq \sum_{j\leq 0} \norm{f_{j}}_{r} \leq \sum_{j\leq 0} 2^{j(\mu-d/p+d/r)} \norm{f_{j}}_{r} \\
		& \leq \sum_{j\leq 0} 2^{j(\mu-d/p)} \norm{(f_{j})_{2^{-j}}}_{r} \\
		&\leq \sum_{j\leq 0} 2^{j(\mu-d/p)} \norm{(f_{j})_{2^{-j}} }_{\mpq} = \sum_{j\leq 0}2^{j\mu} \norm{f_{j}}_{\mpq^{[j]}}.
	\end{align*}
	To take the infimum of the decomposition of $f$, we have $\norm{f}_{r} \lesssim \norm{\mpqmu}.$
	
	Necessity: if we know $\mpqmu \hookrightarrow L^{r}$, then by Remark \ref{rm-mpqmu-mpq}, we have $\mpq \hookrightarrow\mpqmu \hookrightarrow L^{r}$. By Lemma \ref{lem-mpq-lr}, we have $p\vee q \leq r\leq q'$. Meanwhile, when $\si(p,q) <a(p,q)$, we know that $p,q< \infty$. Then we could use the scaling property of $\mpqmu$ (See Proposition 5.5 in \cite{Sugimoto2021Scaling}). 
	
	If we have $\norm{f}_{r} \lesssim \norm{f}_{\mpqmu}$, take $f=\varphi_{\la}$ for some $\varphi \in \sch, 0<\la <1$. Then we have \begin{align*}
		\la^{-d/r} \norm{f}_{r} = \norm{f_{\la}}_{r} \lesssim \norm{f_{\la}}_{\mpqmu} \lesssim \la^{\mu-d/p} \norm{f}_{\mpqmu}.
	\end{align*}
	So, we have $\mu-d/p \leq -d/r.$
\end{proof}

\begin{exa}
	\label{exa:L2-infty}
	For $\mu>0$, there exists $u \in \mtwonemu \setminus L^{2}$. 
	
	For any $j \leq 0, j\in \Z$, denote $k_{j}=(j,\cdots,0) \in \Z^{d}$. Choose $\varphi \in \sch$ with $\supp \widehat{\varphi} \subseteq [-1/8,1/8]^{d}$. Take \begin{align*}
		u= \sum_{j\leq 0} \frac{2^{j(d/2-\mu)}}{j^{2}} e^{ik_{j}x} \varphi(2^{j}x).
	\end{align*}
Then we have \begin{align*}
	\norm{u}_{2} \geq \norm{\Box_{j,k_{j}}u}_{2}  \approx \frac{2^{-j\mu}}{j^{2}}.
\end{align*}
Let $j\rightarrow -\infty$, we have $\norm{u}_{2}=\infty$. 

On the other hand, by the definition of $\mtwonemu$ in Section \ref{sec:pre}, we have \begin{align*}
	\norm{u}_{\mtwonemu} \leq \sum_{j\leq 0} 2^{j\mu} \norm{\frac{2^{j(d/2-\mu)}}{j^{2}} e^{ik_{j}x} \varphi(2^{j}x)}_{M_{2,1}^{[j]}} \lesssim \sum_{j\leq 0} \frac{1}{j^{2}}\lesssim 1.
\end{align*}
\end{exa}

By a similar method to Proposition \ref{prop-mpqmu-embed-lr} with Lemma \ref{lem-mpq-Flr}, we could also characterize the inclusion between $\mpqmu$ and $\F L^{r}$.
\begin{prop}
	\label{prop-mpqmu-Flr}
	Let $1\leq p,q,r\leq \infty, \sigma(p,q) \leq \mu \leq a(p,q), \si(p,q)<a(p,q)$. Then $\mpqmu \hookrightarrow \F L^{r}$ if and only if $q\leq r \leq p', p\leq 2, \mu-d/p\leq -d/r'.$
\end{prop}

\vskip 1.5cm
\section{Linear estimates}\label{sec:linear}
\begin{lemma}[\cite{Kenig1991Oscillatory}]\label{lem-sm-max}
	Denote $D^{s}f = \FF \abs{\xi}^{s} \widehat{f}(\xi)$. Then we have 
	\begin{align*}
		&\norm{D^{3/2} W(t)u}_{L_{x}^{\infty}L_{t}^{2}} \lesssim \norm{u}_{2};\\
		&\norm{ W(t)u}_{L_{x}^{4}L_{t}^{\infty}} \lesssim \norm{u}_{\dot{H}^{1/4}}.
	\end{align*}
\end{lemma}

We will use Lemma \ref{lem-sm-max} to derive some smoothing effect estimates and maximal estimates in local frequency spaces, as follows.

\begin{lemma}
	\label{lem-box-smooth-max}
	For any $p\geq 4, k\in \Z$. We have 
		\begin{align*}
		&\norm{D^{3/2} \Box_{k}W(t)u}_{L_{x}^{\infty}L_{t}^{2}} \lesssim \norm{\Box_{k}u}_{2};\\
		&\norm{\Box_{k} W(t)u}_{L_{x}^{p}L_{t}^{\infty}} \lesssim \norm{\Box_{k}u}_{\dot{H}^{1/p}}.
	\end{align*}
\end{lemma}

\begin{proof}
	Replace $u$ by $\Box_{k}u$ in the above estimates, we could get the first estimate and the second estimate with $p=4$. 
	
	By Bernstein's inequality (Lemma \ref{lem-bernstein}) and $L^{2}$ isometry of $W(t)$, we have \begin{align*}
		\norm{\Box_{k} W(t) u}_{L_{t,x}^{\infty}} \lesssim \norm{\Box_{k} W(t)u}_{L_{t}^{\infty}L_{x}^{2}} = \norm{\Box_{k}u}_{2}.
	\end{align*}
Take the interpolation between $p=4$ and $p=\infty$, we could get the result as desired.
\end{proof}

\begin{prop}\label{prop-homo}
	For any $p\geq 4,k\in \Z$, we have \begin{align*}
		&\norm{\Box_{k}W(t)u}_{L_{x}^{p+2/3} L_{t}^{3p+2}} \lesssim \norm{\Box_{k}u}_{2};\\
		&\norm{\Box_{k}W(t)\pd_{x} u}_{L_{x}^{3p+2} L_{t}^{(3p+2)/(p+1)}} \lesssim \norm{\Box_{k}u}_{2}.
	\end{align*}
\end{prop}
\begin{proof}
	Take the interpolation of the two estimates in Lemma \ref{lem-box-smooth-max}, we have 
	\begin{align*}
		\norm{D^{\al} \Box_{k}W(t)u}_{L_{x}^{a}L_{t}^{b}} \lesssim \norm{\Box_{k}u}_{\dot{H}^{s}},
	\end{align*}
when 
	\ncase{\al&= \frac{3(1-\th)}{2} + 0 \cdot \th, \\
	s&= 0(1-\th) + \th/p,\\
	\frac{1}{a}&=\frac{1-\th}{\infty} + \frac{\th}{p},\\
	\frac{1}{b}&=\frac{1-\th}{2} + \frac{\th}{\infty},}
which is equivalent to 
	\ncase{\th&= p/a,\\
	\al&=3/b,\\
s&= 1/a,\\
1&=\frac{p}{a}+\frac{2}{b}.}
Take $\al=s$ into the equations above, we have $a=p+2/3,b=3p+2$, which is the first estimate in the proposition. Also, taking $\al -s=1$, we can obtain the second estimate as desired.
\end{proof}


By the duality argument, we could get the following estimate of $\mathscr{A} \pd_{x}f$.

\begin{prop}
	\label{prop-nonhomo}
	Let $p \geq 4, k\in \Z$. Then we have \begin{align*}
		\norm{\Box_{k}\mathscr{A} \pd_{x} f}_{L_{x}^{p+2/3} L_{t}^{3p+2}} &\lesssim \norm{\Box_{k}f}_{L_{x}^{\brk{3p+2}'} L_{t}^{\brk{(3p+2)/(p+1)}'}}\\
		&= \norm{\Box_{k}f}_{L_{x}^{\frac{3p+2}{3p+1}} L_{t}^{(3p+2)/(2p+1)}}. 
	\end{align*}
\end{prop}

\begin{proof}
	By the Christ-Kiselev lemma (Lemma \ref{lem-christ}), we only need to prove the following estimate. 
	\begin{align*}
		\norm{  \Box_{k} \int W(t-s)\pd_{x}f(s) ds }_{L_{x}^{p+2/3} L_{t}^{3p+2}} \lesssim \norm{\Box_{k}f}_{L_{x}^{\brk{3p+2}'} L_{t}^{\brk{(3p+2)/(p+1)}'}}.
	\end{align*}
We prove this by duality. 
\begin{align*}
	&\norm{ \Box_{k} \int W(t-s)\pd_{x}f(s) ds }_{L_{x}^{p+2/3} L_{t}^{3p+2}} \\
	&= \sup_{\norm{g}_{L_{x}^{\brk{p+2/3}'} L_{t}^{\brk{3p+2}'}} \leq 1} \inner{\Box_{k}\int W(t-s)\pd_{x} f(s) ds, g(t,x)}_{t,x} \\
	&= \sup_{\norm{g}_{L_{x}^{\brk{p+2/3}'} L_{t}^{\brk{3p+2}'}} \leq 1} \inner{ \Box_{k}\int W(-s) \pd_{x}f(s) ds, \sum_{\abs{\ell}\leq 1} \Box_{k+\ell}\int W(-t) g(t) dt}_{x}\\
	&\leq \sup_{\norm{g}_{L_{x}^{\brk{p+2/3}'} L_{t}^{\brk{3p+2}'}} \leq 1}  \norm{ \Box_{k}\int W(-s) \pd_{x}f(s) ds }_{2} \sum_{\abs{\ell}\leq 1} \norm{\Box_{k+\ell}\int W(-t) g(t) dt}_{2}\\
	&\lesssim \norm{\Box_{k}f}_{L_{x}^{\brk{3p+2}'} L_{t}^{\brk{(3p+2)/(p+1)}'}},
\end{align*}
where the last inequality is the dual version of the estimate in Proposition \ref{prop-homo}.
\end{proof}


To obtain the estimates of $\Box_{j,k}W(t)u, \Box_{j,k} \mathscr{A} f$ from the estimates of $\Box_{k}W(t)u, \Box_{k}\mathscr{A} f$, we need the following scaling lemma.

\begin{lemma}
	\label{lem-scaling}
	Let $0<p,r,p_{1},r_{1},q\leq \infty, 0<\al<\infty,k\in \Z$. If we have \begin{align*}
		&\norm{D^{\al}\Box_{k} W(t) u}_{L_{x}^{p}L_{t}^{r}} \lesssim \norm{\Box_{k} u}_{q},\\
		&\norm{D^{\al} \Box_{k} \mathscr{A} f}_{L_{x}^{p}L_{t}^{r}} \lesssim \norm{\Box_{k} f}_{L_{x}^{p_{1}}L_{t}^{r_{1}}}.
	\end{align*}
Then for any $j\in \Z$, we have \begin{align*}
	&\norm{D^{\al}\Box_{j,k} W(t) u}_{L_{x}^{p}L_{t}^{r}} \lesssim 2^{j\de(p,r,q,\al)}\norm{\Box_{j,k} u}_{q},\\
	&\norm{D^{\al} \Box_{j,k} \mathscr{A} f}_{L_{x}^{p}L_{t}^{r}} \lesssim 2^{j\tau(p,r,p_{1},r_{1},\al)}\norm{\Box_{j,k} f}_{L_{x}^{p_{1}}L_{t}^{r_{1}}},
\end{align*}
where \begin{align*}
	\de(p,r,q,\al)&=\al -4/r-1/p+1/q,\\
	\tau(p,r,p_{1},r_{1},\al)&=\al -4 -4/r -1/p +4/r_{1} +1/p_{1}.
\end{align*} 
\end{lemma}

\begin{proof}
	The proof is based on scaling. For simplicity, we only give the proof of $\mathscr{A}f$.
	
	By definition of $D^{\al},\Box_{j,k},W(t)$ and change of variable, we have \begin{align*}
		D^{\al} \Box_{j,k} \mathscr{A} f (x,t) &= \int_{0}^{t} \FF \abs{\xi}^{\al} \sigma(2^{-j}\xi-k) e^{i(t-s)\abs{\xi}^{4}}\widehat{f}(\xi,s)(x)ds \\
		&= \int_{0}^{t} \FF 2^{j\al} \abs{\xi}^{\al} \sigma(\xi-k) e^{i4^{j}(t-s) \abs{\xi}^{4}} 2^{jd} \widehat{f}(2^{j}\xi,s)(2^{j}x)ds\\
		&= \int_{0}^{2^{4j}t} \FF 2^{j\al} \abs{\xi}^{\al} \sigma(\xi-k) e^{i(2^{4j}t-s) \abs{\xi}^{4}} 2^{jd}2^{-4j} \widehat{f}(2^{j}\xi,2^{-4j}s) (2^{j}x)ds\\
		&= 2^{j\brk{\al-4}} D^{\al} \Box_{k} \mathscr{A}f_{2^{-j}}(2^{j}x,2^{4j}t),
	\end{align*}
	where $f_{2^{-j}} (x,t) = f(2^{-j}x, 2^{-4j}t)$. Similarly, by change of variable, we have $\Box_{k}f_{2^{-j}} (x,t) = \Box_{j,k} f(2^{-j}x,2^{-4j}t)$. 
	Therefore, we have 
	\begin{align*}
		\norm{D^{\al} \Box_{j,k} \mathscr{A} f}_{L_{x}^{p}L_{t}^{r}} &= \norm{2^{j\brk{\al-4}} D^{\al} \Box_{k} \mathscr{A}f_{2^{-j}}(2^{j}x,2^{4j}t) }_{L_{x}^{p}L_{t}^{r}}\\
		&= 2^{j\brk{\al-4}}2^{-4j/r} 2^{-j/p} \norm{D^{\al} \Box_{k} \mathscr{A}f_{2^{-j}}}_{L_{x}^{p}L_{t}^{r}}\\
		&\lesssim 2^{j(\brk{\al-4-1/p-4/r})} \norm{\Box_{k} f_{2^{-j}}}_{L_{x}^{p_{1}}L_{t}^{r_{1}}}\\
		&= 2^{j\tau(p,r,p_{1},r_{1},\al)} \norm{\Box_{j,k}f}_{L_{x}^{p_{1}}L_{t}^{r_{1}}}.
	\end{align*}
\end{proof}

\begin{rem}
	Notice that when $j\leq 0, \de,\tau\geq 0$ or $j\geq 0, \de,\tau\leq 0$, the estimates above are better than the estimates of only taking place of $u$ by $\Box_{j,k}u$ in the conditions.
\end{rem}

By Propositions \ref{prop-homo}, \ref{prop-nonhomo} and the scaling lemma above, we have 
\begin{prop}
	\label{prop-boxjkWtAf}
	Let $p\geq 4, j\leq 0,k \in \Z $. We have \begin{align*}
		&\norm{\Box_{j,k}W(t)u}_{L_{x}^{p+2/3} L_{t}^{3p+2}} \lesssim2^{jA} \norm{\Box_{j,k}u}_{2},\\
		&\norm{\pd_{x} \Box_{j,k} \mathscr{A} f }_{L_{x}^{p+2/3} L_{t}^{3p+2}} \lesssim 2^{jB} \norm{\Box_{j,k}f}_{L_{x}^{\brk{3p+2}'} L_{t}^{\brk{(3p+2)/(p+1)}'}},
	\end{align*}
where $A= \frac{3(p-4)}{2(3p+2)}, B= \frac{2(p-4)}{3p+2}$.
\end{prop}

Recall the work spaces $X^{\mu}(A)$ in Definition \ref{def-work-space}. By Proposition \ref{prop-boxjkWtAf}, we could obtain the main estimates of $W(t)$ as follows.
\begin{prop}
	\label{prop-linear-estimate}
	Let $p\geq 4, \mu \in \real, A= \frac{3(p-4)}{2(3p+2)}, B= \frac{2(p-4)}{3p+2}$. Then we have 
	\begin{align*}
	&\norm{W(t)u_{0}}_{X^{\mu}(L_{x}^{p+2/3} L_{t}^{3p+2})} \lesssim \norm{u_{0}}_{\mtwone^{\mu+A}};\\
	&\norm{\pd_{x} \mathscr{A}f}_{X^{\mu}(L_{x}^{p+2/3} L_{t}^{3p+2})}	 \lesssim \norm{f}_{X^{\mu+B}(L_{x}^{\brk{3p+2}'} L_{t}^{\brk{(3p+2)/(p+1)}'})}.
	\end{align*}
\end{prop}

\vskip 1.5cm
\section{Nonlinear estimates} \label{sec:nonlinear}

In this section, we give the nonlinear estimate of the work space $X^{\mu}(L_{x}^{p} L_{t}^{\ga})$. First, we need two lemmas.

\begin{lemma}
	\label{lem-Xl1-eembedding}
	Let $\ell \leq j \leq 0, 1\leq p,\ga \leq \infty$. Then we have 
	\begin{align*}
		\sum_{k\in \Z} \norm{\Box_{j,k} u}_{L_{x}^{p} L_{t}^{\ga}} \lesssim \sum_{k\in \Z} \norm{\Box_{\ell,k} u}_{L_{x}^{p} L_{t}^{\ga}}. 
	\end{align*}
\end{lemma}

\begin{proof}
	For any $k_{1}\in \Z,$ denote $\wedge_{k_{1},j}= \set{k\in \Z: \Box_{\ell,k} \Box_{j,k_{1}} \neq 0}$. Then we know that \begin{align*}
		\Box_{j,k_{1}} u = \sum_{k\in \wedge_{k_{1},j}} \Box_{\ell,k} \Box_{j,k_{1}} u.
	\end{align*}
Meanwhile, $\set{\wedge_{k_{1},j}}_{k_{1}\in \Z}$ are finitely overlapped. In fact, if there exists $k\in \wedge_{k_{1},j} \cap \wedge_{k_{2},j}$, then we have $\abs{2^{j}k_{1} - 2^{j} k_{2}} \leq 2^{j}+2^{\ell}$, which means that $\abs{k_{1}-k_{2}} \lesssim 1$.

Therefore, by the triangular inequality, we have 
\begin{align*}
	\sum_{k_{1}\in \Z}\norm{\Box_{j,k_{1}} u}_{L_{x}^{p} L_{t}^{\ga}} &\leq \sum_{k_{1}\in \Z} \sum_{k\in \wedge_{k_{1},j}} \norm{\Box_{\ell,k} \Box_{j,k_{1}} u}_{L_{x}^{p} L_{t}^{\ga}} \\
	&\lesssim \sum_{k_{1}\in \Z} \sum_{k\in \wedge_{k_{1},j}} \norm{\Box_{\ell,k}u}_{L_{x}^{p} L_{t}^{\ga}}\lesssim \sum_{k\in \Z} \norm{\Box_{\ell,k}}_{L_{x}^{p} L_{t}^{\ga}}.
\end{align*}
\end{proof}

\begin{lemma}
	\label{lem-Xmu-embedding}
	Let $1\leq p_{1} \leq \ga\leq p\leq \infty,\mu\in \real$. Then we have $$X^{\mu+1/p_{1}}(L_{x}^{p_{1}} L_{t}^{\ga}) \hookrightarrow  X^{\mu+1/p} (L_{x}^{p} L_{t}^{\ga}) .$$
\end{lemma}

\begin{proof}
	For any $u=\sum_{j\leq 0} u_{j}$, by Definition \ref{def-work-space}, Minkowski's inequality and Bernstein's inequality (Lemma \ref{lem-bernstein}), we have 
	\begin{align*}
		\norm{u}_{X^{\mu+1/p}(L_{x}^{p} L_{t}^{\ga})} &\leq \sum_{j\leq 0} 2^{j(\mu+1/p)} \sum_{k\in \Z} \norm{\Box_{j,k} u}_{L_{x}^{p} L_{t}^{\ga}} \leq \sum_{j\leq 0} 2^{j(\mu+1/p)} \sum_{k\in \Z} \norm{\Box_{j,k} u}_{ L_{t}^{\ga}L_{x}^{p}} \\
		& \leq \sum_{j\leq 0} 2^{j(\mu+1/p_{1})} \sum_{k\in \Z} \norm{\Box_{j,k} u}_{ L_{t}^{\ga}L_{x}^{p_{1}}} \leq  \sum_{j\leq 0} 2^{j(\mu+1/p_{1})} \sum_{k\in \Z} \norm{\Box_{j,k} u}_{ L_{x}^{p_{1}}L_{t}^{\ga}}.
	\end{align*}
So, we have \begin{align*}
	\norm{u}_{X^{\mu+1/p}(L_{x}^{p} L_{t}^{\ga})} \leq \norm{u}_{X^{\mu+1/p_{1}}(L_{x}^{p_{1}} L_{t}^{\ga})},
\end{align*}
which is equivalent to $X^{\mu+1/p_{1}}(L_{x}^{p_{1}} L_{t}^{\ga}) \hookrightarrow  X^{\mu+1/p} (L_{x}^{p} L_{t}^{\ga}) $.
\end{proof}

\begin{prop}
	\label{prop-nonlinear-estimate-general}
	Let $1\leq p,p_{1},p_{2},\ga,\ga_{1},\ga_{2} \leq \infty, \mu \in \real$, satisfying \begin{align*}
		\frac{1}{p} = \rev{p_{1}}+\rev{p_{2}},\quad \rev{\ga} = \rev{\ga_{1}} +\rev{\ga_{2}}.
	\end{align*}
Then we have \begin{align*}
	\norm{uv}_{X^{\mu}(L_{x}^{p} L_{t}^{\ga})} \lesssim \norm{u}_{X^{\mu}(L_{x}^{p_{1}} L_{t}^{\ga_{1}})} \norm{v}_{X^{0}(L_{x}^{p_{2}} L_{t}^{\ga_{2}})} +\norm{u}_{X^{0}(L_{x}^{p_{1}} L_{t}^{\ga_{1}})} \norm{v}_{X^{\mu}(L_{x}^{p_{2}} L_{t}^{\ga_{2}})}.
\end{align*}
\end{prop}

\begin{proof}
The proof is based on the proof of Lemma 8.8 in \cite{Sugimoto2021Scaling}. 

For any $u=\sum_{j\leq 0} u_{j}, v=\sum_{\ell\leq 0} v_{\ell}$, we have \begin{align*}
	uv= \sum_{j,\ell\leq 0} u_{j} v_{\ell} &= \sum_{j\leq 0} \sum_{\ell\leq j} u_{j}v_{\ell} + \sum_{\ell\leq 0} \sum_{j<\ell} u_{j}v_{\ell}\\
	&:= \rmnum{1} +\rmnum{2}.
\end{align*}
By the triangular inequality, we have \begin{align} \label{eq-uv-1+2}
	\norm{uv}_{X^{\mu}(L_{x}^{p} L_{t}^{\ga})}\leq \norm{\rmnum{1}}_{X^{\mu}(L_{x}^{p} L_{t}^{\ga})} +\norm{\rmnum{2}}_{X^{\mu}(L_{x}^{p} L_{t}^{\ga})}.
\end{align} 
By Definition \ref{def-work-space} and the triangular inequality, we have 
\begin{align}\label{eq-1}
	\norm{\rmnum{1}}_{X^{\mu}(L_{x}^{p} L_{t}^{\ga})} &\leq \sum_{j\leq 0} 2^{j\mu} \sum_{k\in \Z} \norm{\Box_{j,k} \brk{\sum_{\ell\leq j} u_{j}v_{\ell}}}_{L_{x}^{p} L_{t}^{\ga}} \notag\\
	& \leq \sum_{j\leq 0} 2^{j\mu}\sum_{\ell\leq j} \sum_{k\in \Z} \norm{\Box_{j,k}(u_{j}v_{\ell})}_{L_{x}^{p} L_{t}^{\ga}}
\end{align}

Take decomposition of $u_{j} = \sum_{k_{1}\in \Z} \Box_{j,k_{1}} u_{j}, v_{\ell} =\sum_{k_{2}\in \Z} \Box_{j,k_{2}} v\ell$, by the orthogonality of $\Box_{j,k}\brk{\Box_{j,k_{1}} u_{j} \Box_{j,k_{2}} v_{\ell}}$, we have 
\begin{align}\label{eq-2}
	\eqref{eq-1}&\leq \sum_{j\leq 0} 2^{j\mu}\sum_{\ell\leq j} \sum_{k\in \Z} \sum_{k_{1},k_{2}\in \Z }\norm{\Box_{j,k}(\Box_{j,k_{1}}u_{j} \Box_{j,k_{2}}v_{\ell})}_{L_{x}^{p} L_{t}^{\ga}} \notag \\
	&\leq \sum_{j\leq 0} 2^{j\mu}\sum_{\ell\leq j}  \sum_{k_{1},k_{2}\in \Z } \sum_{k:\abs{k-k_{1}-k_{2}} \lesssim 1}\norm{\Box_{j,k}(\Box_{j,k_{1}}u_{j} \Box_{j,k_{2}}v_{\ell})}_{L_{x}^{p} L_{t}^{\ga}}	\notag\\
	&\leq \sum_{j\leq 0} 2^{j\mu}\sum_{\ell\leq j}  \sum_{k_{1},k_{2}\in \Z } \norm{\Box_{j,k_{1}}u_{j} \Box_{j,k_{2}}v_{\ell}}_{L_{x}^{p} L_{t}^{\ga}}
\end{align}
By H\"older's inequality and Lemma \ref{lem-Xl1-eembedding}, we have 
\begin{align*}
	\eqref{eq-2} &\leq \sum_{j\leq 0} 2^{j\mu}\sum_{\ell\leq j}  \sum_{k_{1},k_{2}\in \Z } \norm{ \Box_{j,k_{1}}u_{j}}_{L_{x}^{p_{1}} L_{t}^{\ga_{1}}} \norm{ \Box_{j,k_{2}}v_{\ell}}_{L_{x}^{p_{2}} L_{t}^{\ga_{2}}}\\
	&= \sum_{j\leq 0} 2^{j\mu} \sum_{k_{1}\in \Z} \norm{\Box_{j,k_{1}} u_{j}}_{L_{x}^{p_{1}} L_{t}^{\ga_{1}}} \sum_{\ell\leq j} \sum_{k_{2}\in \Z} \norm{\Box_{j,k_{2}} v_{\ell}}_{L_{x}^{p_{2}} L_{t}^{\ga_{2}}}\\
	&\leq \sum_{j\leq 0} 2^{j\mu} \sum_{k_{1}\in \Z} \norm{\Box_{j,k_{1}} u_{j}}_{L_{x}^{p_{1}} L_{t}^{\ga_{1}}} \sum_{\ell\leq j} \sum_{k_{2}\in \Z} \norm{\Box_{\ell,k_{2}} v_{\ell}}_{L_{x}^{p_{2}} L_{t}^{\ga_{2}}}.
\end{align*}
Therefore, we have \begin{align*}
	\norm{\rmnum{1}}_{X^{\mu}(L_{x}^{p} L_{t}^{\ga})} \lesssim \norm{u}_{X^{\mu}(L_{x}^{p_{1}} L_{t}^{\ga_{1}})} \norm{v}_{X^{0}(L_{x}^{p_{2}} L_{t}^{\ga_{2}})}.
\end{align*}
By the same method, we also have \begin{align*}
	\norm{\rmnum{2}}_{X^{\mu}(L_{x}^{p} L_{t}^{\ga})} \lesssim \norm{u}_{X^{0}(L_{x}^{p_{1}} L_{t}^{\ga_{1}})} \norm{v}_{X^{\mu}(L_{x}^{p_{2}} L_{t}^{\ga_{2}})}.
\end{align*}
Combine these two estimates, we could get the estimate as desired in the proposition.
\end{proof}
By induction, we could obtain the corollary as follows. 
\begin{cor}
	\label{cor-nonlinear}
	Let $1\leq m\in \N, 1\leq p,p_{i},\ga,\ga_{i}\leq \infty$, for $i=1,\cdots,m+1$. If \begin{align*}
		\rev{p} = \sum_{i=1}^{m+1} \rev{p_{i}}, \quad \rev{\ga} = \sum_{i=1}^{m+1} \rev{\ga_{i}}.
	\end{align*} Then we have 
	\begin{align*}
		\norm{u^{m+1}}_{X^{0}(L_{x}^{p}L_{t}^{\ga})} \lesssim \prod _{i=1}^{m+1} \norm{u}_{X^{0}(L_{x}^{p_{i}}L_{t}^{\ga_{i}})}.
	\end{align*}
\end{cor}

Applying this corollary to the work space $$X^{0}(L_{x}^{\brk{3p+2}'} L_{t}^{\brk{(3p+2)/(p+1)}'})=X^{0}(L_{x}^{(3p+2)/(3p+1)} L_{t}^{(3p+2)/(2p+1)}),$$ we could get  the estimate as follows.

\begin{prop}\label{prop-nonlinear-estimate}
	Let $m=2p \in \N,p \geq 4$. Then we have \begin{align*}
		\norm{u^{m+1}}_{X^{0}(L_{x}^{(3p+2)/(3p+1)} L_{t}^{(3p+2)/(2p+1)})} \lesssim \norm{u}_{X^{0}(L_{x}^{p+2/3}L_{t}^{3p+2})}^{m+1}.
	\end{align*}
\end{prop}

\begin{proof}
	Denote \begin{align*}
		(a,b)=\Cas{(n,2), & $m=2n\in 2\N$; \\
		(n,5), & $m=2n+1 \in 2\N +1$.}
	\end{align*}
Then we have \begin{align*}
	\frac{3p+1}{3p+2} &= \frac{3a}{3p+2} + \frac{b}{2(3p+2)} +\frac{m+1-a-b}{\infty};\\
	\frac{2p+1}{3p+2} &= \frac{a}{3p+2} +\frac{b}{3p+2} + \frac{m+1-a-b}{3p+2}.
\end{align*}
By Corollary \ref{cor-nonlinear}, we have 
\begin{align}
	\label{eq-nonlinear1}
	&\norm{u^{m+1}}_{X^{0}(L_{x}^{(3p+2)/(3p+1)} L_{t}^{(3p+2)/(2p+1)})} \notag\\
	&\lesssim \norm{u}_{X^{0}(L_{x}^{p+2/3}L_{t}^{3p+2})}^{a} \norm{u}_{X^{0}(L_{x}^{2(3p+2)}L_{t}^{3p+2})}^{b} \norm{u}_{X^{0}(L_{x}^{\infty}L_{t}^{3p+2})}^{m+1-a-b}.
\end{align}
By Lemma \ref{lem-Xmu-embedding} and Remark \ref{rm-Xmu-embed}, we have 
\begin{align*}
	&\norm{u}_{X^{0}(L_{x}^{2(3p+2)}L_{t}^{3p+2})} \lesssim \norm{u}_{X^{5/(6p+4)}(L_{x}^{p+2/3}L_{t}^{3p+2})} \lesssim \norm{u}_{X^{0}(L_{x}^{p+2/3}L_{t}^{3p+2})},\\
	&\norm{u}_{X^{0}(L_{x}^{\infty}L_{t}^{3p+2})} \lesssim \norm{u}_{X^{3/(3p+2)}(L_{x}^{p+2/3}L_{t}^{3p+2})} \lesssim \norm{u}_{X^{0}(L_{x}^{p+2/3}L_{t}^{3p+2})}.
\end{align*}
To take these estimates into \eqref{eq-nonlinear1}, we could obtain the nonlinear estimate as desired.

\end{proof}

\vskip 1.5cm
\section{Proof of Theorem \ref{thm-GWP}}\label{sec:proof}
\begin{proof}
  Denote $p=m/2, A= \frac{3(p-4)}{2(3p+2)}, B= \frac{2(p-4)}{3p+2} $, let $ 0\leq \mu \leq A$, for any $u_{0}\in \mtwonemu$, define the operator \begin{align*}
  \mathscr{T}: X^{\mu-A}(L_{x}^{p+2/3}L_{t}^{3p+2}) \ &\longrightarrow \ X^{\mu-A}(L_{x}^{p+2/3}L_{t}^{3p+2})\\
  u \ &\longrightarrow \ W(t)u_{0} + \pd_{x}\mathscr{A} u^{m+1}.
  \end{align*}
By Proposition \ref{prop-linear-estimate}, we have \begin{align*}
	\norm{\mathscr{T}u}_{X^{\mu-A}(L_{x}^{p+2/3}L_{t}^{3p+2}) } \lesssim \norm{u_{0}}_{X^{\mu}(L_{x}^{p+2/3}L_{t}^{3p+2}) } + \norm{u^{m+1}}_{X^{\mu-A+B}(L_{x}^{\brk{3p+2}'} L_{t}^{\brk{(3p+2)/(p+1)}'}) }.
\end{align*}
Notice that $\mu-A\leq 0, \mu-A+B\geq 0$, so by Remark \ref{rm-Xmu-embed} and Proposition \ref{prop-nonlinear-estimate}, we have \begin{align*}
	\norm{u^{m+1}}_{X^{\mu-A+B}(L_{x}^{\brk{3p+2}'} L_{t}^{\brk{(3p+2)/(p+1)}'}) } &\leq \norm{u^{m+1}}_{X^{0}(L_{x}^{\brk{3p+2}'} L_{t}^{\brk{(3p+2)/(p+1)}'}) } \\
	&\lesssim \norm{u}_{X^{0}(L_{x}^{p+2/3}L_{t}^{3p+2})}^{m+1} \\
	&\leq \norm{u}_{X^{\mu-A}(L_{x}^{p+2/3}L_{t}^{3p+2})}^{m+1}.
\end{align*}
Therefore, by the standard contraction mapping argument, we see that \eqref{eq-d4nls} has a unique solution $u\in X^{\mu-A}(L_{x}^{p+2/3}L_{t}^{3p+2})$. Moreover, since $W(t)$ is bounded on $\mtwonemu$ and $\mtwonemu$ is a multiplication algebra (Proposition 6.2 in \cite{Sugimoto2021Scaling}), it can be shown that $u\in C(\real,\mtwonemu)$. 
\end{proof}

\vskip 1.5cm
\section*{Acknowledgments}
I would like to thank my advisor Baoxiang Wang for inspiring ideas and discussions.

\bibliographystyle{acm}
\bibliography{scalingpde}
\newpage
%
%

\end{document}